\documentclass[12pt]{amsart}
\pdfoutput=1
\usepackage{euler}
\usepackage{amsmath}
\usepackage{amsfonts}
\usepackage{amssymb}
\usepackage{amsthm}
\usepackage{amscd}
\usepackage{graphicx}
\usepackage[width=13.50cm,
left=3.00cm, right=3.00cm
]{geometry}

\usepackage{lpic}
\usepackage{longtable}
\usepackage{url}

\theoremstyle{plain}
\newtheorem{theorem}{Theorem}
\newtheorem{proposition}[theorem]{Proposition}
\newtheorem{lemma}[theorem]{Lemma}
\newtheorem{corollary}[theorem]{Corollary}

\theoremstyle{definition}

\newtheorem{remark}[theorem]{Remark}

\title[{Triple-crossing number, the genus of a knot or link and torus knots}]{Triple-crossing number, the genus of a knot\\ or link and torus knots}

\author{Micha{\l} Jab{\l}onowski}

\address{Institute of Mathematics, Faculty of Mathematics, Physics and Informatics,\\ University of Gda\'nsk, 80-308 Gda\'nsk, Poland}



\email{michal.jablonowski@gmail.com}

\begin{document}

\maketitle

\begin{abstract}
We show that the triple-crossing number of any knot is greater or equal to twice its (canonical) genus and we show an even stronger bound in the case of links. As an application we show that this bound is strong enough to obtain the triple-crossing numbers of all torus knots, and of many more knots and their connected sums.
\end{abstract}

\section{Introduction}

In general position of planar diagrams of knots and links, two strands meet at every crossing. It is known since \cite{Ada13} that any knot and any link has a diagram where, at each of its multiple points in the plane, exactly three strands are allowed to cross (pairwise transversely). Such triple-point diagrams have been studied in several recent papers, such as \cite{Ada13, ACFIPVWZ14, AHP19, JabTro20, Nis19}.
\par
The triple-crossing number of a knot or a link $L$, denoted $c_3(L)$, is defined analogously to the classical (double-crossing) number as the least number of triple-crossings for any triple-crossing diagram of $L$. There are lower bounds for the triple-crossing number: in terms of double-crossing number $c_3(L)\geq \frac{1}{3}c_2(L)$, and if $L$ is an alternating link then $c_3(L)\geq \frac{1}{2}c_2(L)$ (see \cite{Ada13}); in terms of double-crossing braid index $c_3(L)\geq\beta_2(L)-1$ (where $L$ is an nonsplit link, see \cite{Nis19}).
\par
There is the known inequality between the genus $g$ of a link $L$, the canonical genus $g_c$ and its double-crossing number, namely $c_2(L)\geq 2\cdot g_c(L)\geq 2\cdot g(L)$. Combining the above-mentioned inequalities, one sees that $c_3(L)\geq \frac{2}{3}\cdot g_c(L)$. We will improve this lower bound as follows (where $r(L)$ is the number of components of the link $L$).

\begin{theorem}\label{t1}
	Let $L$ be a knot or a link. Then $c_3(L)\geq 2\cdot g_c(L)+r(L)-1.$
\end{theorem}

The paper is organized as follows. In Section\;\ref{s2} of this paper we give a lower bound for the triple-crossing number of a given link in terms of the canonical genus (and hence also the genus) of that link. As an application we show in Section\;\ref{s3} that the triple-crossing number of torus $T(p,q)$ knot is equal to $(p-1)(q-1)$ and describe its minimal triple-crossing diagram. In Section\;\ref{s4} we conclude from the bounds the triple-crossing number of additional knots.

\section*{Acknowledgments}

The author would like to thank the anonymous referee who provided useful and detailed comments on a earlier version of the manuscript.

\section{Definitions}\label{s1}

The \emph{projection} of a knot or a link $L\subset \mathbb{R}^3$ is its image under the standard projection $\pi:\mathbb{R}^3\to\mathbb{R}^2$ such that it has only a finite number of self-intersections (one may have to isotope $L$ in other to achieve this), called \emph{multiple points}, and in each multiple point each pair of its strands are transverse. 
\par
If each multiple point of a projection has multiplicity three then we call this projection a \emph{triple-crossing projection}. The \emph{triple-crossing} is a three-strand crossing with the strand labeled $T, M, B$, for top, middle and bottom. 
The \emph{triple-crossing diagram} is a triple-crossing projection such that each of its triple points is a triple-crossing (i.e. it has the labelings at each triple-point), such that $\pi^{-1}$ of the strand labeled $T$ (in the neighborhood of that triple point) is on the top of the strand corresponding to the strand labeled $M$, and the latter strand is on the top of the strand corresponding to the strand labeled $B$ (see Figure\;\ref{r01}).

	\begin{figure}[h!t]
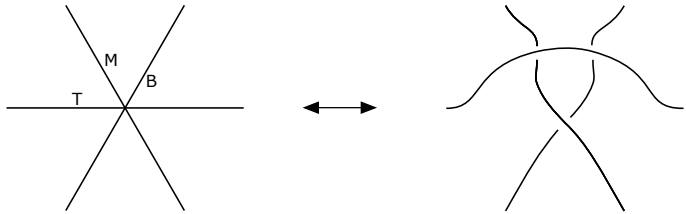

		\centering		
		\begin{lpic}[]{./figs/LM01(9cm)}
		
		\end{lpic}
		\caption{A deconstruction/construction of a triple-crossing.}
		\label{r01}
	\end{figure}

The \emph{triple-crossing number} of a knot or link $L$, denoted $c_3(L)$, is the least number of triple-crossings for any triple-crossing diagram of $L$. The classical double-crossing number invariant we will denote by $c_2$. A \emph{minimal triple-crossing diagram} of a knot or link $L$ is a triple-crossing diagrams of $L$ that has exactly $c_3(L)$ triple-crossings.
\par
A \emph{piecewise natural orientation} (see \cite{AHP19}) on a triple-crossing diagram is an orientation of each component of that link, such that in each crossing the strands are oriented in-out-in-out-in-out, as we encircle the crossing. 
\par
Given any knot or link, one can obtain a \emph{Seifert surface} (i.e. an orientable surface embedded into $\mathbb{R}^3$ whose boundary equals the knot or link) by applying Seifert’s algorithm to an oriented double-crossing of its projection. We split the diagram at each crossing and reglue so that the orientations of the resulting strands match. Then each of the disjoint simple closed curves that result is spanned with a disk. Half-twisted bands are attached to the boundaries of the disk at each crossing. We call a surface obtained in this manner a \emph{canonical Seifert surface} for the knot or link.

\section{Canonical genus and triple-crossing number of a knot}\label{s2}

The \emph{canonical genus} of a link $L$, denoted here by $g_c(L)$, is the least genus of any canonical Seifert surface for $L$. The \emph{genus} of a link $L$, denoted by $g(L)$, is the least genus of any Seifert surface for $L$. This definition and Theorem\;\ref{t1} immediately imply the lower bound of the triple-crossing number $c_3$ by replacing $g_c$ with $g$ in the right-hand side of the inequality, since for any link $L$ we have $g_c(L)\geq g(L)$.

	\begin{figure}[h!t]
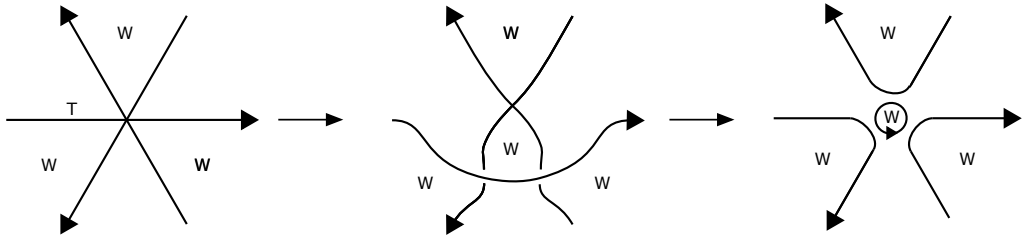

	\centering
	
	\begin{lpic}[]{./figs/LM02(13.5cm)}
		
	\end{lpic}
	\caption{Deconstructing a triple-crossing and performing the resolution.}
	\label{r04}
\end{figure}

\begin{proof}[Proof of Theorem\;\ref{t1}]
	Let $L$ be any (unoriented) knot or link and let $D$ be a minimal triple-crossing diagram of $L$, and let $n$ be the number of the set of triple-crossings of $D$ (denoted $vertices(D)$). The set of edges of $D$, denoted $edges(D)$, contains $3n$ elements. For a planar diagram we know, from Euler characteristic, that the number of elements in $faces(D)$ (consisting of \emph{regions} defined as connected components of the planar complement of $D$) equals $\# faces(D)=2+\# edges-\# vertices=2+3n-n=2n+2$. We can color elements of $faces(D)$ in a checkerboard fashion (see \cite[Lemma 2.1.]{AHP19}). We now pick an orientation of $L$ that corresponds to the piecewise natural orientation of $D$ induced by the checkerboard coloring of regions of $D$, such that the boundary of each white region is oriented clockwise (i.e. the piecewise natural orientation with respect to this checkerboard coloring). Without loss of generality assume that at least half of elements from $faces(D)$ are colored white and are elements of the set $whitefaces(D)$, i.e. $\# whitefaces(D)\geq \frac{1}{2}(2n+2)=n+1.$
	\par
	We now deconstruct every triple-crossing to a three double-crossings in a way such that we (locally) increase the number of white faces by one (see Figure\;\ref{r04} where the white regions are marked by the letter $W$, and the flat crossing can have arbitrary crossing information). This figure covers all cases since, the diagram $D$ has the piecewise natural orientation. In the end we obtain a double-crossing diagram $\overline{D}$ such that $vertices(\overline{D})=3n$, and $\# white faces(\overline{D})\geq (n+1)+n= 2n+1$.
	\par
	We can prove that $\# whitefaces(\overline{D})=\#\text{Seifert-circles}(\overline{D})$. To show this notice that this checkerboard coloring of $\overline{D}$ is induced from a checkerboard coloring of $D$, and that the orientation of $D$ is the piecewise natural one with respect to this checkerboard coloring. The Seifert algorithm performed on $\overline{D}$ gives us a canonical Seifert surface $F$ of genus $\displaystyle g(F)=\frac{1}{2}(2+\#vertices(\overline{D})-\#\text{Seifert-circles}(\overline{D})-r(L))\leq\frac{1}{2}(2+3n-(2n+1)-r(L))=\frac{n-r(L)+1}{2}$,
	 where $r(L)$ is the number of components of $L$. 
	 \par
	 Hence, $c_3(L)=n\geq 2\cdot g(F)+r(L)-1\geq 2\cdot g_c(L)+r(L)-1$.	
\end{proof}

\section{Torus knots and their minimal triple-crossing diagrams}\label{s3}

Torus knots, denoted $T(p,q)$ for coprime positive integers $p<q$ are known to have their minimal double-crossing diagrams as a closure of the braid word $(\sigma_1\sigma_2\cdots\sigma_{p-1})^q$ (see \cite{Mur91}). In this infinite family of knots, up to now, the triple-crossing number of nontrivial torus knots has been known only for the trefoil and the cinquefoil.

\begin{proposition}
	Let $K$ be a torus $T(p,q)$ knot (for coprime positive integers $p,q$). Then $c_3(K)=(p-1)(q-1)$.
\end{proposition}

\begin{proof}
	It is well-known (see \cite{BurZie03}) that the genus of the torus knot $T(p,q)$ equals $g(T(p,q))=\frac{1}{2}(p-1)(q-1)$. From Theorem\;\ref{t1} we have now that $c_3(T(p,q))\geq 2\cdot g_c(T(p,q))\geq 2\cdot g(T(p,q))=(p-1)(q-1)$. To finish the proof it suffices to find a triple-crossing diagram of $T(p,q)$ with $(p-1)(q-1)$ triple-crossings. Assuming $p<q$, we do this by first putting $p-1$ circles covering all double-crossings, one between each neighboring strands of the closure of $(\sigma_1\sigma_2\cdots\sigma_{p-1})^q$ braid word. Then performing the folding operation on each circle (described in \cite{Ada13}), obtaining triple-crossing diagram of $T(p,q)$ knot with $(p-1)q-(p-1)=(p-1)(q-1)$ triple-crossings. 
\end{proof}

\begin{remark}
	By analogy, the same statement as in the above proposition holds for the mirror image of those knots and their diagrams. 
\end{remark}

The above proposition can be proven at the following level of generality.

\begin{proposition}
Suppose that $K$ is a knot that may be written as the closure of a braid $\beta$ on $p$ strands, such that $\beta \in B_p$ may be written as a word of length $k$ in the standard generators of $B_p$ (in particular this means that $\beta$ is a positive braid). Then we have
$c_3(K) = k - p + 1.$
\end{proposition}

\begin{proof}
The upper bound on $c_3(K)$ comes from \cite[Theorem 2.2.]{Ada13}, and the lower bound on $c_3(K)$ comes from Theorem\;\ref{t1}, together with the fact that Seifert's algorithm, applied to a $2$-crossing diagram obtained by taking the closure of a positive braid, yields a Seifert surface of minimal genus (see \cite{Cro89}).
\end{proof}

\section{Other knots}\label{s4}

We have the following upper bound on triple-crossing number.

\begin{lemma}[\cite{Ada13}]\label{l2}
	Let $K$ be a nontrivial knot or a nontrivial link. If $K\not=T(2,n)$ for any $n\in\mathbb{Z}$, then $ c_3(K)\leq c_2(K)-2.$	
\end{lemma}	

\begin{remark}\label{rem:01}

From Lemma\;\ref{l2}, Theorem\;\ref{t1} and the table of knot genus from \cite{LivMoo20} we obtain, because of equality of upper- and lower-bounds, exact values of (up to now unknown) triple-crossing numbers for many more knots, such as:\\ $8_{2}$, $8_{5}$, $8_{7}$, $8_{9}$, $8_{10}$, $8_{16}$, $8_{17}$, $8_{18}$, $10_{2}$, $10_{5}$, $10_{9}$, $10_{17}$, $10_{46}$, $10_{47}$, $10_{48}$, $10_{62}$, $10_{64}$, $10_{79}$, $10_{82}$, $10_{85}$, $10_{91}$, $10_{94}$, $10_{99}$, $10_{100}$, $10_{104}$, $10_{106}$, $10_{109}$, $10_{112}$, $10_{116}$, $10_{118}$, $10_{123}$, $10_{139}$, $10_{152}$, $K12a_{146}$, $K12a_{369}$, $K12a_{576}$, $K12a_{716}$, $K12a_{722}$, $K12a_{805}$, $K12a_{815}$, $K12a_{819}$, $K12a_{824}$, $K12a_{835}$, $K12a_{838}$, $K12a_{850}$, $K12a_{859}$, $K12a_{864}$, $K12a_{869}$, $K12a_{878}$, $K12a_{898}$, $K12a_{909}$, $K12a_{916}$, $K12a_{920}$, $K12a_{981}$, $K12a_{984}$, $K12a_{999}$, $K12a_{1002}$, $K12a_{1011}$, $K12a_{1013}$, $K12a_{1027}$, $K12a_{1047}$, $K12a_{1051}$, $K12a_{1114}$, $K12a_{1120}$, $K12a_{1128}$, $K12a_{1134}$, $K12a_{1168}$, $K12a_{1176}$, $K12a_{1191}$, $K12a_{1199}$, $K12a_{1203}$, $K12a_{1209}$, $K12a_{1210}$, $K12a_{1211}$, $K12a_{1212}$, $K12a_{1214}$, $K12a_{1215}$, $K12a_{1218}$, $K12a_{1219}$, $K12a_{1220}$, $K12a_{1221}$, $K12a_{1222}$, $K12a_{1223}$, $K12a_{1225}$, $K12a_{1226}$, $K12a_{1227}$, $K12a_{1229}$, $K12a_{1230}$, $K12a_{1231}$, $K12a_{1233}$, $K12a_{1235}$, $K12a_{1238}$, $K12a_{1246}$, $K12a_{1248}$, $K12a_{1249}$, $K12a_{1250}$, $K12a_{1253}$, $K12a_{1254}$, $K12a_{1255}$, $K12a_{1258}$, $K12a_{1260}$, $K12a_{1273}$, $K12a_{1283}$, $K12a_{1288}$, $K12n_{242}$, $K12n_{472}$, $K12n_{574}$, $K12n_{679}$, $K12n_{688}$, $K12n_{725}$, $K12n_{888}$.
\end{remark}

\begin{corollary}
	Let $K_1$ and $K_2$ be each either torus knot or a knot from Remark\;\ref{rem:01}. Then $c_3(K_1)+c_3(K_2)=c_3(K_1\#K_2)$. 
\end{corollary}

\begin{proof}
	From Theorem\;\ref{t1} and the additivity of the genus (in respect to the connected sum $\#$) we have: $c_3(K_1)+c_3(K_2)\geq c_3(K_1\#K_2)\geq 2\cdot g(K_1\#K_2)=2\cdot g(K_1)+2\cdot g(K_2)=c_3(K_1)+c_3(K_2).$
\end{proof}

\end{document}